\documentclass{amsart}

\usepackage[utf8]{inputenc}
\usepackage{verbatim}
\usepackage{graphicx}
\usepackage{graphicx,caption2,psfrag,float,color}
\usepackage{amssymb}
\usepackage{amscd}
\usepackage{amsmath}
\usepackage{mathrsfs}

\newtheorem{theorem}{Theorem}[section]

\newtheorem{lemma}[theorem]{Lemma}

\numberwithin{equation}{subsection}
\newtheorem{definition}[theorem]{Definition}

\title{Explicit estimates of sums related to the Nyman-Beurling criterion for the Riemann Hypothesis}
\author{Helmut Maier and Michael Th. Rassias}
\date{\today}
\address{Department of Mathematics, University of Ulm, Helmholtzstrasse 18, 89081 Ulm, Germany.}
\email{helmut.maier@uni-ulm.de}
\address{Institute of Mathematics, University of Zurich, CH-8057, Zurich, Switzerland }
\email{michail.rassias@math.uzh.ch}\thanks{}

\begin{document}

\maketitle
 
\begin{abstract} 
We give an estimate for sums appearing in the Nyman-Beurling criterion for the Riemann
Hypothesis. These sums contain the M\"obius function and are related to the imaginary part
of the Estermann zeta function. The estimate is remarkably sharp in comparison to other sums containing the M\"obius function. The bound is smaller than the trivial bound - essentially the number of terms - by a fixed power of that number. The exponent is made explicit. The methods intensively use tools from the theory of continued fractions and from the theory of 
Fourier series.

\newline

\end{abstract}
\section{Introduction}
According to the approach of Nyman-Beurling-B\'aez-Duarte (see \cite{baez1}, \cite{BEC}) to the Riemann Hypothesis, the Riemann Hypothesis is true if and only if 
$$\lim_{N\rightarrow \infty} d_N^2=0\:,$$
where
\[
d_N^2:=\inf_{D_N}\frac{1}{2\pi}\int_{-\infty}^\infty\left|1-\zeta D_N\left(\frac{1}{2}+it\right)\right|^2\frac{dt}{\frac{1}{4}+t^2}\tag{1.1}
\]
and the infimum is over all Dirichlet polynomials 
$$D_N(s):=\sum_{n=1}^{N}\frac{a_n}{n^s}\:,\ a_n\in\mathbb{C}\:,$$
of length $N$ (see \cite{bcf}).\\
Various authors (cf. \cite{baez1}, \cite{baez2}, \cite{burnol}) have investigated the question, which 
asymptotics hold for $d_N$ if the Riemann Hypothesis is true.\\
S. Bettin, J. B. Conrey and D. W. Farmer \cite{bcf} show the following:\\
If the Riemann Hypothesis is true and if 
$$\sum_{|Im(\rho)|\leq T}\frac{1}{|\zeta'(\rho)|^2}\ll T^{\frac{3}{2}-\delta}$$
for some $\delta>0$, then
$$\frac{1}{2\pi}\int_{-\infty}^{\infty}\left|1-\zeta V_N\left(\frac{1}{2}+it\right)\right|^2\frac{dt}{\frac{1}{4}+t^2}\sim\frac{2+\gamma-\log4\pi}{\log N}$$
for 
\[
V_N(s):=\sum_{n=1}^N\left(1-\frac{\log n}{\log N}\right)\frac{\mu(n)}{n^s}\:.\tag{1.2}
\]
Moreover, it follows from \cite{bcf} that under certain assumptions among all Dirichlet polynomials $D_N(s)$ the infimum in (1.1) is assumed for $D_N(s)=V_N(s)$. It thus is of interest to obtain an unconditional estimate for the integral in (1.1).\\
If we expand the square in (1.1) we obtain
\begin{align*}
d_N^2&=\inf_{D_N}\Bigg(\int_{-\infty}^\infty\left(1-\zeta\left(\frac{1}{2}+it\right)D_N\left(\frac{1}{2}+it\right)-\zeta\left(\frac{1}{2}-it\right)\overline{D}_N\left(\frac{1}{2}+it\right)\right)\frac{dt}{\frac{1}{4}+t^2}\\
&\ \ \ +\int_{-\infty}^\infty\left|\zeta\left(\frac{1}{2}+it\right)\right|^2\left|D_N\left(\frac{1}{2}+it\right)\right|^2\frac{dt}{\frac{1}{4}+t^2}\Bigg)\:.
\end{align*}
The last integral evaluates as 
$$\sum_{1\leq h,k\leq N} a_h\bar{a_k} h^{-1/2}k^{-1/2} \int_{-\infty}^\infty\left|\zeta\left(\frac{1}{2}+it\right)\right|^2\left(\frac{h}{k}\right)^{it}\frac{dt}{\frac{1}{4}+t^2}\:.$$
We have
\begin{align*}
b_{h,k}&:=\frac{1}{2\pi\sqrt{hk}}\int_{-\infty}^\infty\left|\zeta\left(\frac{1}{2}+it\right)\right|^2\left(\frac{h}{k}\right)^{it}\frac{dt}{\frac{1}{4}+t^2}\\
&=\frac{\log2\pi-\gamma}{2}\left(\frac{1}{h}+\frac{1}{k}\right)+\frac{k-h}{2hk}\log\frac{h}{k}-\frac{\pi}{2hk}\left(V\left(\frac{h}{k}\right)+V\left(\frac{k}{h}\right)\right)\:,\tag{1.3}
\end{align*}
where
$$V\left(\frac{h}{k}\right):=\sum_{m=1}^{k-1}\left\{\frac{mh}{k}\right\}\cot\left(\frac{\pi m h}{k}\right)$$
is Vasyunin's sum (see \cite{VAS}).\\
It can be shown that 
\[
V\left(\frac{h}{k}\right)=-c_0\left(\frac{\bar{h}}{k}\right),\tag{1.4}
\]
where $h\bar{h}\equiv 1(\bmod k)$, with the cotangent sum
$$c_0\left(\frac{h}{k}\right):=\sum_{l=1}^{k-1}\frac{l}{k}\cot\left(\frac{\pi h l}{k}\right)\:,$$
where $h,k\in\mathbb{N}$, $k\geq 2$, $1\leq h\leq k$, $(h,k)=1$ (see \cite{BEC}, \cite{mr}).

Ishibashi \cite{ISH} observed that $c_0$ is related to the value at $s=0$ or $s=1$ 
of the Estermann zeta function by the functional equation of the \textit{imaginary part}. Namely
$$c_0\left(\frac{a}{q}\right)=\frac{1}{2}D_{sin}\left(0,\frac{a}{q}\right)=2q\pi^{-2}D_{\sin}\left(1,\frac{\bar{a}}{q}\right)\:,$$
where for $x\in\mathbb{R}$, $Re(s)>1$, we have:
$$D_{\sin}(s,x):=\sum_{n=1}^\infty\frac{d(n)\sin(2\pi nx)}{n^s}$$
and $a\bar{a}\equiv 1(\bmod q)$.\\
If $x\in\mathbb{R}\setminus\mathbb{Q}$, then Wilton \cite{wilton} showed that the convergence of the above series at $s=1$ is equivalent to the convergence of 
$$\sum_{n\geq 1}(-1)^n\frac{\log v_{n+1}}{v_n}\:,$$
where $u_n/v_n$ denotes the $n$-th partial quotient of $x$.\\
The irrational numbers for which this sum converges are called \textbf{Wilton numbers}.

A basic ingredient in the papers \cite{mr2}, \cite{mr5} have been the representations of 
Balazard, Martin in their papers \cite{balaz1}, \cite{balaz2}, of the function 
$$g(x):=\sum_{l\geq 1}\frac{1-2\{lx\}}{l}\:,$$
involving the Gauss transform from the theory of continued fractions as well as from the paper \cite{Marmi}, by Marmi, Moussa and Yoccoz.\\
One can show (\cite{bre}) that $g(x)$ can also be written in the form
$$-\sum_{n\geq 1}\frac{d(n)}{\pi n}\:\sin(2\pi n x)=-\frac{1}{\pi}\: D_{sin}(1, x)\:.$$
These concepts and results also play an important role in the present paper. We shall represent them in the next section.\\
\textit{Remark.} There is another, simpler, approach to our result. By \cite{bettin} (see 1st
display of page 5) $V(h/k)$ is approximately
$$\sum_{n\leq k^2}\frac{d(n)}{n}\:\sin\left(2\pi n \:\frac{h}{k}\right)\:.$$
One then may use well-known estimates for exponential sums containing the M\"obius function 
(cf. \cite{IKW}, formula (13.49)).

\section{Statement of result}
We prove the following theorem:
\begin{theorem}\label{main}
Let $D\geq 2$. Let $C$ be the number which is uniquely determined by
$$C\geq \frac{\sqrt{5}+1}{2},\ 2C-\log C - 1- 2\log 2=\frac{1}{2}\log 2\:.$$
Let $v_0$ be determined by
$$v_0\left(1-\left(1+2\log 2\left(C+\frac{\log 2}{2}\right)^{-1}\right)^{-1}+2+\frac{4}{\log 2} C\right)=2\:.$$
Let $z_0:=2-\left(2+\frac{4}{\log 2} C\right) v_0$. Then for all $\epsilon >0$ we have
$$\sum_{k^D\leq n< 2k^D} \mu(n)g\left(\frac{n}{k}\right)\ll_{\epsilon} k^{D-z_0+\epsilon}\:.$$
\end{theorem}

\section{Continued Fractions}
In this section we provide some basic facts related to the theory of continued fractions. We also recall facts and results from the papers \cite{balaz2}.
\begin{definition}\label{def21}
Let $X:=[0,1]\setminus\mathbb{Q}$. For $x\in(0,1)$ we set $\alpha(x):=\{1/x\}$. For $x\in X$ we define recursively
$$\alpha_0(x):=x,\ \alpha_l(x):=\alpha(\alpha_{l-1}(x)),\ \ \text{for $l\in\mathbb{N}$}\:.$$
This definition is also valid for $x\in\mathbb{Q}$ and $l\in\mathbb{N}$, whenever $\alpha_{l-1}(x)\neq 0$. We set
$$a_l(x):=\left[\frac{1}{\alpha_{l-1}(x)}\right]\:.$$
\end{definition}
We have:
\begin{lemma}\label{lem22}
For $x\in X$, let
$$x=[0;a_1(x),\ldots, a_l(x),\ldots]$$
be the continued fraction expansion of $x$.\\
For $x\in(0,1)\cap \mathbb{Q}$ we have 
$$x=[0;a_1(x),\ldots, a_L(x)]\:,$$
where $L$ is the last $l$, for which $\alpha_{l-1}\neq 0$.\\
We define the partial quotient of $p_l(x), q_l(x)$ by 
$$\frac{p_l(x)}{q_l(x)}:=[0;a_1(x),\ldots, a_l(x)],\ where\ (p_l(x), q_l(x))=1\:.$$
We have
$$p_{l+1}=a_{l+1}p_l+p_{l-1}$$
$$q_{l+1}=a_{l+1}q_l+q_{l-1}\:.$$
\end{lemma}
\begin{proof}
(cf. \cite{hens}, p. 7.)
\end{proof}
\begin{definition}\label{def23}
For $r\in(0,1)\cap\mathbb{Q}$ let
$$r=[0;a_1(r),\ldots, a_L(r)]\:.$$
Then we call $L$ the \textbf{depth} of $r$. 
\end{definition}
\begin{definition}\label{def24}
Let $x\in X$. Then for $l\in\mathbb{N}_0:=\mathbb{N}\cup\{0\}$, we set:
$$\beta_l(x):=\alpha_0(x)\alpha_1(x)\cdots \alpha_l(x)$$
(by convention $\beta_{-1}=1$) and
$$\gamma_l(x):=\beta_{l-1}(x)\log\frac{1}{\alpha_l(x)},\ \text{where}\ l\geq 0,$$
so that $\gamma_0(x)=\log(1/x).$\\
Let $r\in(0,1)\cap\mathbb{Q}$ be a rational number of depth $L$. Then we set:
$$\beta_l(r):=\alpha_0(r)\alpha_1(r)\ldots \alpha_l(r)\:,\ \ \text{for $0\leq l\leq L$ }$$
and
$$\gamma_l(r):=\beta_{l-1}(r)\log\frac{1}{\alpha_l(r)}\:,\ \ \text{for $0\leq l\leq L$.}$$
For $x\in X$ we define Wilton's function $\mathcal{W}(x)$ by 
$$\mathcal{W}(x):=\sum_{l\geq 0}(-1)^l\gamma_l(x)\:,\ \ \text{for all $x\in X$,}$$
for which the series is convergent.
\end{definition}
\begin{definition}\label{def25}(definitions from Sec. 4.1 of \cite{balaz2})$ $\\
For $\lambda\geq 0$ let 
$$A(\lambda):=\int_0^\infty\{t\}\{\lambda t\}\frac{dt}{t^2}\:.$$
For $x>0$ let
$$Q(x):=\frac{x+1}{2}A(1)-A(x)-\frac{x}{2}\log x\:.$$
For $x\in X$ let
$$G(x):=\sum_{j\geq 0}(-1)^j\beta_{j-1}Q(\alpha_j(x))\:.$$
For a rational number $r$ of depth $L$ let
$$G(r):=\sum_{j\leq L}(-1)^j \beta_{j-1}(r)Q(\alpha_j(r))\:.$$
\end{definition}
\begin{lemma}\label{lem26}
We have
$$\sup_{\lambda\geq 1}|A(\lambda+h)-A(\lambda)|\leq \frac{1}{2}h\log\left(\frac{1}{h}\right)+O(h)\:,\ \ (0<h< 1)\:.$$
\end{lemma}
\begin{proof}
This is Proposition 28 of \cite{balaz2}.
\end{proof}
\begin{definition}\label{def27}
Let 
\begin{eqnarray}
\delta(x):=\left\{ 
  \begin{array}{l l}
   0\: & \quad \text{, if $x\in X$}\vspace{2mm}\\ 
    \frac{(-1)^{L+1}A(1)}{2q}\: & \quad \text{, if $x=p/q\in[0,1]$, $(p,q)=1$, $x$ of depth $L$}\:.\\
  \end{array} \right.
\nonumber
\end{eqnarray}
\end{definition}
\begin{lemma}\label{lem27}
The series $g(x)$ and $\mathcal{W}(x)$ converge for the same values $x\in[0,1]$ and we have 
$$g(x)=\mathcal{W}(x)-2G(x)-2\delta(x)$$
in each point of convergence. 
\end{lemma}
\begin{proof}
This is Proposition 28 of \cite{balaz2}.
\end{proof}
\begin{definition}\label{def28}
For $s\in\mathbb{N}$, $x\in X$ or $x$ a rational number with depth $\geq s$, 
we define 
$$\mathcal{L}(x, s):=\sum_{\nu=0}^s(-1)^\nu(T^\nu l)(x)\:,$$
where $l(x):=\log(1/x)$ and the operator $T\::\: L^p\rightarrow L^p$ is defined
by $$Tf(x):=xf(\alpha(x))\:.$$
We write
$$g(x)=:g_{sm}(x, s)+g_{sing}(x,s)\:,$$
where
$$g_{sm}(x, s)=\mathcal{L}(x, s)\:.$$
We call $g_{sm}$ the ``smooth part" and $g_{sing}$ the ``singular part" of $g$.
\end{definition}
\begin{lemma}\label{lem29}
Let $x\in X$ be a Wilton number or a rational number. Then we have:
$$\mathcal{W}(x)=l(x)-x\mathcal{W}(\alpha(x))\:.$$
\end{lemma}
\begin{proof}
This follows directly from the definition of Wilton's function $\mathcal{W}(x)$.
\end{proof}
\begin{lemma}\label{lem210}
For $s\in\mathbb{N}$, $x\in X$ or $x$ a rational number with depth 
$\geq s+1$ we have 
$$\mathcal{L}(x,s)=\mathcal{W}(x)-(-1)^{s+1}T^{s+1}\mathcal{W}(x)\:.$$
\end{lemma}
\begin{proof}
This follows from Lemma \ref{lem29} by the same computation as in the
proof of Lemma 2.10 in \cite{mr5}, which is valid also for $x$ a rational number with depth $\geq s+1$.
\end{proof}
\begin{lemma}\label{lem211}
For $s\in\mathbb{N}_0$, $x\in X$ or $x$ a rational number of depth $\geq s+1$ we have 
$$\alpha_s(x)\alpha_{s+1}(x)\leq \frac{1}{2}\:.$$
\end{lemma}
\begin{proof}
This is Lemma 2.11 of \cite{mr5}, whose proof is also valid for rational $x$ of depth $\geq s+1$. 
\end{proof}
\begin{definition}\label{def212}
Let $\mathcal{E}$ be a measurable subset of $(0,1)$. The measure $m$ 
is defined by 
$$m(\mathcal{E}):=\frac{1}{\log 2}\int_{\mathcal{E}}\frac{dx}{1+x}\:.$$
\end{definition}
\begin{lemma}\label{lem213}
The measure $m$ is invariant with respect to the map $\alpha$, i.e. 
$$m(\alpha(\mathcal{E}))=m(\mathcal{E})$$
for all measurable subsets $\mathcal{E}\subseteq (0,1)$.
\end{lemma}
\begin{proof}
This result is well-known.
\end{proof}
\begin{lemma}\label{lem213}
Let $s\in \mathbb{N}$, $p>1$. For $f\in L^p$, we have 
$$\int_0^1|T^s f(x)|^p dm(x)\leq g^{(s-1)p}\int_0^1|f(x)|^pdm(x)\:,$$
where 
$$g:=\frac{\sqrt{5}-1}{2}<1\:.$$
\end{lemma}
\begin{proof}
For the proof of this result, due to Marmi, Moussa and Yoccoz \cite{Marmi}, see \cite{mr5} Lemma 2.8, (ii).
\end{proof}
\begin{definition}\label{def215}
Let $s\in\mathbb{N}$, $b_0=0$ and $b_1,\ldots, b_s\in\mathbb{N}$. The \textbf{cell} of depth $s$, $\mathcal{C}(b_1,\ldots, b_s)$ is the interval with the endpoints $[0;b_1,\ldots, b_s]$ and \mbox{$[0;b_1,\ldots, b_{s-1},b_s+1]$.}
\end{definition}
\begin{lemma}\label{lem216}
In the interior of the cell $\mathcal{C}(b_1,\ldots, b_s)$ of depth $s$, the 
functions $a_j, p_j, q_j$ are constants for $j\leq s$,
$$a_j(x)=b_j,\ \frac{p_j(x)}{q_j(x)}=[0; b_1,\ldots, b_j],\ (x\in\mathcal{C}(b_1,\ldots, b_j))\:.$$
The endpoints of $\mathcal{C}(b_1,\ldots, b_s)$ are 
\[
\frac{p_s}{q_s}\ \ \text{and}\ \ \frac{p_s+p_{s-1}}{q_s+q_{s-1}}\:.\tag{3.1}
\]
For $x\in X$ and $s\in\mathbb{N}$ there is a unique cell of depth $s$
that contains $x$.\\
Within a cell of depth $s$ we have the derivatives
$$\alpha_s'=(-1)^s(q_s+\alpha_sq_{s-1})^2\:,$$
$$\gamma_s'=(-1)^sq_{s-1}\log\left(\frac{1}{\alpha_s}\right)+(-1)^{s-1}/\beta_s\:.$$
We also have 
$$\beta_s(x)=\frac{1}{q_{s+1}(x)+\alpha_{s+1}(x)q_s(x)}\ \ (x\in X)\:.$$
\end{lemma}
\begin{proof}
See \cite{balaz2}, sections 2.2, 2.3 and 2.4.
\end{proof}
\begin{lemma}\label{lem216b}
Let $\mathcal{C}=\mathcal{C}(b_1,\ldots, b_s)$ a cell of depth $s$, $x_1, x_2\in\mathcal{C}$, $d=|x_2-x_2|$. Then we have:
$$|g(x_1)-g(x_2)|\ll d^{2}q_{s+1}\log s+2^{-s/2}\log k\:.$$
\end{lemma}
\begin{proof}
Let without loss of generality $x_1<x_2$. From Definition \ref{def28} and the
formula for $\gamma'_s$ in Lemma \ref{lem216} it follows by integration over the
interval $[x_1, x_2]$ that 
\[
|g_{sm}(x_1)-g_{sm}(x_2)|\ll d^2 q_{s+1}\log s\:.\tag{3.2}
\]
From Lemma \ref{lem211} we obtain
\[
|g_{sing}(x_i)|\leq 2^{-s/2}\log k\ \ (i=1, 2).\tag{3.3}
\]
The result follows from (3.2) and (3.3).
\end{proof}

\begin{lemma}\label{lem217}
Let $\mathcal{C}(b_1,\ldots, b_s)$ be as in Definition \ref{def215}. Then we have
$$\log\text{meas}(\mathcal{C}(b_1,\ldots, b_s))\leq -2\sum_{j=1}^s \log b_j\:.$$
\end{lemma}
\begin{proof}
The cell $\mathcal{C}(b_1,\ldots, b_s)$ has the length
\[
\frac{1}{q_s(q_s+q_{s-1})}\:.\tag{3.4}
\]
From $q_j=b_{j+1}q_j+q_{j-1}$ it follows by induction that
\[
q_j\geq \prod_{i\leq j} b_i\:.  \tag{3.5}
\]
From (3.4) and (3.5) it follows that
$$\text{meas}(\mathcal{C}(b_1,\ldots, b_s))\leq  \prod_{j\leq s}  b_j^{-2}\:, $$
which concludes the proof of Lemma \ref{lem217}.
\end{proof}
\begin{lemma}\label{lem218}
We have $$\log q_s\leq 2\sum_{1\leq j\leq s}\log b_j+s\log 2\:.$$
\end{lemma}
\begin{proof}
This follows from 
$$q_{j+1}=b_{j+1}q_j+q_{j-1}\leq 2b_{j+1}q_j\:.$$
\end{proof}
\begin{lemma}\label{lem219}
Let $\epsilon>0$, $C_1\geq\frac{\sqrt{5}+1}{2}$, $C_2=\frac{1}{2}C_1-\log C_1-1+\log 2$. Then we have for $s\geq s(\epsilon)$ sufficiently large: 
$$meas\{x\in(0,1)\::\: q_s(x)\geq \exp(C_1s)\}\leq \exp(-(C_2-\epsilon)s)\:.$$
\end{lemma}
\begin{proof} 
The smallest sequence $(q_s)$ is obtained for $b_j=1$ for all $j\leq s$, which 
gives $q_s=F_s$, $F_s$ being the $s$-th Fibonacci number. This shows that 
the assumption $C_1\geq\frac{\sqrt{5}+1}{2}$ does not imply any loss of generality.\\
By partition into cells of depth $s$ we obtain from Lemmas \ref{lem217} and \ref{lem218}
the following:
\begin{align*}
&meas\{x\in(0,1)\::\: q_s(x)\geq \exp(C_1s)\}\leq\sum_{\substack{\vec{b}=(b_1,\ldots,b_s)\in\mathbb{N}^s\::\\\sum_{1\leq j\leq s}\log b_j\geq \frac{1}{2}C_1s}}meas(\mathcal{C}(b_1,\ldots, b_{s}))\\
&\ \ \ \ \ \ \ \leq \sum_{\substack{\vec{b}\in\mathbb{N}^s\::\\ \sum_{1\leq j\leq s}\log b_j\geq \frac{1}{2}C_1s}}\left(\int_{b_1}^{b_1+1}[u_1]^{-2}\:du_1  \right)\cdots \left(\int_{b_s}^{b_s+1}[u_s]^{-2}\:du_s  \right)\\
&\ \ \ \ \ \ \ \leq 2^s \int\displaylimits_{\substack{[1,\infty)^s\::\\ \sum_{1\leq j\leq s}\log u_j\geq \frac{1}{2}C_1s}} (u_1\cdots u_s)^{-2}\: du_1\cdots du_s\\
&\ \ \ \ \ \ \ =2^s \int\displaylimits_{{\mathbb{R}_+^s\::\:v_1+\cdots+v_s\geq \frac{1}{2}C_1s}} \exp(-v_1-\cdots -v_s)\: dv_1\cdots dv_s\:. \tag{3.6}
\end{align*}
The integral in (3.6) is 
$$Prob\left(X_1+\cdots+X_s\geq\frac{1}{2}C_1s\right)\:,$$
where $X_1,\ldots, X_s$ are i.i.d. exponentially distributed random variables with rate 1. The distribution of the sum $X_1+\cdots+X_s$ is the Gamma 
distribution with density 
\[
f(x,s,1)=\frac{x^{s-1}e^{-x}}{(s-1)!}\tag{3.7}
\]
Lemma \ref{lem219} now follows from (3.6), (3.7) and Stirling's formula.
\end{proof}
\vspace{7mm}
\section{Vaughan's identity}
We now express the M\"obius function by Vaughan's identity.
\begin{lemma}\label{lem31}
Let $w>1$. For $n\in\mathbb{N}$ we have:
$$\mu(n)=c_1(n)+c_2(n)+c_3(n)\:,$$
where 
$$c_1(n):=\sum_{\substack{\alpha\beta\gamma=n\\\alpha\geq w,\: \beta\geq w}}\mu(\gamma)c_4(\alpha)c_4(\beta)\:,$$
for
$$c_4(\alpha):=-\sum_{\substack{(d_1,d_2)\::\: d_1d_2=\alpha\\ d_1\leq w}}\mu(d_1)\:.$$
\begin{eqnarray}
c_2(n):=\left\{ 
  \begin{array}{l l}
   2\mu(n)\: & \quad \text{, if $n\leq w$}\vspace{2mm}\\ 
    0\: & \quad \text{, if $n>w$}\\
  \end{array} \right.
\nonumber
\end{eqnarray}
and
$$c_3(n):=-\sum_{\substack{\alpha\beta\gamma=n\\\alpha\leq w,\: \beta\leq w}}\mu(\alpha)\mu(\beta)\:.$$
\end{lemma}
\begin{proof}
We introduce the generating Dirichlet series
$$\sum_{n=1}^{+\infty}c_1(n)n^{-s}=\frac{1}{\zeta(s)}E(s)^2$$
$$\sum_{n=1}^{+\infty}c_2(n)n^{-s}=2M_w(s)$$
$$\sum_{n=1}^{+\infty}c_3(n)n^{-s}=-M_w(s)^2\zeta(s)$$
where
$$M_w(s):=\sum_{n\leq w}\mu(n)n^{-s}\:,$$
and
$$E(s):=\zeta(s)M_w(s)-1\:.$$
The proof follows by comparison of the coefficients of these Dirichlet series.
\end{proof}
%
We now obtain 
$$\sum_{k^D\leq n<2k^D}\mu(n)g\left(\frac{n}{k}\right)={\textstyle\sum}_{1}+{\textstyle\sum}_{2}+{\textstyle\sum}_{3}\:,$$
where
\[
{\textstyle\sum}_{i}:=\sum_{k^D\leq n<2k^D}c_i(n)g\left(\frac{n}{k}\right)\:.\tag{4.1}
\]
From the definitions for $c_1$ and $c_4$ in Lemma \ref{lem31} we obtain:
$${\textstyle\sum}_{1}=\sum_{\substack{s\geq w,\: t\geq w \\ k^D\leq st\gamma<2k^D}}\mu(\gamma)\sum(s,t)\:,$$
where
$$\sum(s,t) := \sum_{\substack{d_1d_2=s \\ d_1\leq w}}\mu(d_1)\sum_{\substack{e_1e_2=t \\ e_1\leq w}}\mu(e_1)g\left(\frac{d_1d_2e_1e_2\gamma}{k}\right)\:.$$
We now choose 
$$w=k^{2\delta_0},\ \text{for}\ \delta_0>0\ \text{fixed}\:,$$ 
to be determined later. We also fix $v_0>0$ to be determined later as well.\\
We partition the sum ${\textstyle\sum}_{1}$ as follows:
\[
{\textstyle\sum}_{1}={\textstyle\sum}_{1,1}+{\textstyle\sum}_{1,2}\tag{4.2}
\]
with
\[
{\textstyle\sum}_{1,1}:=\sum_{\substack{s\geq w,\: t\geq w \\ k^D\leq st\gamma<2k^D\\st\geq k^{4\delta_0+4v_0}}}\mu(\gamma)\sum(s,t)\tag{4.3}
\]
and
\[
{\textstyle\sum}_{1,2}:=\sum_{\substack{s\geq w,\: t\geq w\\ k^D\leq st\gamma<2k^D\\st< k^{4\delta_0+4v_0}}}\mu(\gamma)\sum(s,t)\:.\tag{4.4}
\]
\vspace{7mm}
\section{The sum ${\textstyle\sum}_{1,1}$}
We further partition the sum ${\textstyle\sum}_{1,1}$. The condition 
$$d_1e_1d_2e_2\geq k^{4\delta_0+4v_0}$$ 
implies that $d_2e_2\geq k^{4v_0}$ and therefore at least one of the following
two cases must hold:\\
1) $d_2\geq k^{2v_0}$\\
2) $e_2\geq k^{2v_0}\:.$

\noindent By symmetry we may restrict ourselves to the case $e_2\geq k^{2v_0}$.\\
For $v\geq v_0$ we define the partial sum ${\textstyle\sum}_{1,1}^{(v)}$ of
${\textstyle\sum}_{1}$ by
\[
{\textstyle\sum}_{1,1}^{(v)}:=\sum_{\substack{s\geq w,\: t\geq w\\ k^D\leq st\gamma<2k^D\\k^{2v}\leq e_2<2k^{2v} }}\mu(\gamma)\sum_{\substack{d_1d_2=s\\ d_1\leq w}}\mu(d_1)\sum_{\substack{e_1e_2=t\\ e_1\leq w}}\mu(e_1)g\left(\frac{d_1d_2e_1e_2\gamma}{k} \right)  \tag{5.1}
\]
We then partition the sum ${\textstyle\sum}_{1,1}$ into $O(\log k)$ partial sums of 
the form ${\textstyle\sum}_{1,1}^{(v)}$.
\[
{\textstyle\sum}_{1,1}:={\textstyle\sum}_{1,1}^{(v_0)}+{\textstyle\sum}_{1,1}^{(v_1)} +\cdots+{\textstyle\sum}_{1,1}^{(v_z)}\:, \tag{5.2}
\]
where $v_0<v_1<\cdots< v_z$. In the sequel we assume $v\geq v_0$ fixed and
estimate the sum ${\textstyle\sum}_{1,1}^{(v)}$.\\
We fix $\delta_1=\delta_1(v)>0$, $\delta_2=\delta_2(v)>0$
to be determined later. We construct an appropriate Diophantine approximation
of $l/k$.  By the Theorem of Dirichlet there are positive integers $a$, $q=q(l)$, 
$(a, q)=1$, $q\leq k^{\delta_2}$, such that 
\[
\left|  \frac{l}{k}-\frac{a}{q} \right|\leq (qk^{\delta_2})^{-1}\:.\tag{5.3}
\]
\begin{definition}\label{defn4141}
For $\vec{d}=(d_1, d_2, e_1, \gamma)$ let $l(\vec{d})$ be the uniquely defined 
integer $l$ with $0<l\leq k$, $l\equiv \Pi(\vec{d}\:)\: \bmod k$, where $\Pi(\vec{d}\:)=d_1d_2e_1\gamma$. We define the exceptional set
$$\mathcal{E}(v, \delta_1, \delta_2, k):=\{\vec{d}=(d_1, d_2, e_1, \gamma)\::\:\exists e_2\in [k^{2v}, 2k^{2v}),\ k^D\leq \Pi(\vec{d}\:)e_2<2k^D\}\:.$$
\end{definition}
\begin{lemma}\label{lem4141}
We have
$$|\mathcal{E}(v, \delta_1, \delta_2, k)|=O_\epsilon(k^{D-2v+\delta_1-\delta_2+\epsilon})$$
for all $\epsilon >0$.
\end{lemma}
\begin{proof}
Let $\vec{d}\in \mathcal{E}(v, \delta_1, \delta_2, k)$, $l= l(\vec{d})$ and $q(l)<k^{\delta_1}$. Then there is an interval
$$N(a, q)=\left(\frac{a}{q}-\left(qk^{\delta_1}\right)^{-1},\ \frac{a}{q}+\left(qk^{\delta_1}\right)^{-1}\right) $$
of length $2(qk^{\delta_1})^{-1}$ with $l/k\in N(a, q)$.\\
$N(a, q)$ contains $\ll k^{1-\delta_1}q^{-1}$ numbers $l/k$. The set
$$\mathcal{F}:=\{l(\vec{d})k^{-1}\::\: \vec{d}\in \mathcal{E}(v, \delta_1, \delta_2, k)\}$$
is contained in the union
$$\bigcup_{(a, q)\::\: q\leq k^{\delta_1}}N(a, q)\:.$$
There are $O(k^{2\delta_1})$ pairs $(a, q)$. Thus 
$$|\mathcal{F}|=O(k^{1+\delta_1-\delta_2})\:.$$
The interval $[\frac{1}{2}k^{D-2v},\ 2k^{D-2v})$ can be partitioned into
$O(k^{D-1-2v})$ abetting subintervals of length $\leq k$. \\
The result now follows from well known estimates for the divisor function $d_4$.
\end{proof}
\begin{lemma}\label{lem4242}
For all $\epsilon>0$ we have
$$\sum_{\vec{d}\in\mathcal{E}(v, \delta_1, \delta_2, k)}\left| \sum_{\substack{e_2\::\: k^{2v}\leq e_2< 2k^{2v}\\ k^D\leq \Pi(\vec{d}\:)e_2<2k^D}}g\left(\frac{\Pi(\vec{d}\:)e_2}{k}  \right) \right| =O_\epsilon\left( k^{D+\delta_1-\delta_2+\epsilon} \right)\:.  $$
\end{lemma}
\begin{proof}
This follows from Lemma \ref{lem4141} and from the bound $g(h/k)=O(\log k)$.
\end{proof}
We fix $s_1=s_1(v)>0$, $C_4>0$ to be determined later and set
\[
C_5:=\frac{1}{2}\:C_4-\log C_4-1+\log 2\:.\tag{5.4}
\]
By Lemma \ref{lem219} all but $O(k\exp(-(C_5-\epsilon)s_1))$ values of $l/k$ belong to 
cells  $\mathcal{C}$ of depth $s_1$ with denominator 
$$q_{s_1}(l):=q_{s_1}(\mathcal{C})\leq \exp(C_4s_1)\:.$$
\begin{definition}\label{def4242}
We call the cell $\mathcal{C}$ of depth $s_1$ \textbf{exceptional}, if
$q_{s_1}(\mathcal{C})>\exp(C_4s_1)$, otherwise non-exceptional.
\end{definition}
\begin{definition}\label{def4343}
For $k^{\delta_1}\leq q< k^{\delta_2}$ let
\[
\mathcal{A}(q):=\left\{\vec{d}\::\: \exists \: r\in \mathbb{Z},\ \left|\frac{l(\vec{d}\:)}{k}-\frac{r}{q}\right|<\frac{1}{q^2}\right\}\:.  \tag{5.5}
\]
\end{definition}

We now use the Diophantine approximation property for $\vec{d}\in\mathcal{A}(q)$
to estimate the sum
$$\sum_{\substack{e_2\::\: k^{2v}\leq e_2< 2k^{2v}\\ k^D\leq \Pi(\vec{d}\:)e_2<2k^D}}g\left(\frac{\Pi(\vec{d}\:)e_2}{k}  \right) \:.$$
\begin{definition}\label{def4444}
Let
$$\mathcal{J}:=\left[\max(k^{2v}, k^D\Pi(\vec{d}\:)^{-1},\ \min(2k^{2v}, 2k^D\Pi(\vec{d}\:)^{-1})\right)\:.$$ 
\end{definition}
We partition $\mathcal{J}$ into $j_0-1=O(k^{2v}q^{-1})$ abetting subintervals 
$I_j=[e_{2,j},\ e_{2, j}+q)$ of length $q$ and an additional interval $I_{j_0}=[e_{2,j_0},\ e_{2, j_0}+\tilde{q})$ of length $\tilde{q}<q$: 
$$\mathcal{J}=\bigcup_{j=0}^{j_0} I_j,\ e_{2, j+1}=e_{2, j}+q\ \ \text{for $0\leq j\leq j_0-2$.}$$
For each $j$ with $0\leq j\leq j_0-1$ we construct a function $\Phi_j$ as follows:\\
The elements of the sequence $(r(e_{2,j}+g))_{g=0}^{q-1}$ form a complete residue system $\bmod\: q$. Thus there is a $g_0\in \mathbb{N}_0$, such that
$$ \left| \left\{ \frac{\Pi(\vec{d}\:)}{k}\:e_{2,j} \right\}-\left\{ \frac{r(e_{2, j}+g_0)}{q} \right\} \right| \leq \frac{1}{2q}\:. $$
From the inequality (4.4) we then obtain:
\[
 \left| \left\{ \frac{\Pi(\vec{d}\:)(e_{2,j}+g)}{k} \right\}-\left\{ \frac{r(e_{2, j}+g_0+g)}{q} \right\} \right| \leq \frac{2}{q}\:. \tag{5.6}
 \]
 We then define 
 $$\Phi_j(\vec{d}, e_{2, j}+g)=\left\{ \frac{r}{q}(e_{2,j}+g_0+g) \right\}\:.$$
 \begin{lemma}\label{lem4343}
 Let $\vec{d}\in \mathcal{A}(q)$, $I_j$ $(j\leq j_0-1)$ as described above. Then
 the sequence $(q\Phi_j(\vec{d}\:, e_{2,j}+g))$ forms a complete residue system 
 $\bmod\: q$.
 \end{lemma}
\begin{proof}
This is contained in the observations above.
\end{proof} 
\begin{definition}\label{defn4545}
Let $\vec{d}\:\in\mathcal{A}(q)$, $e_2\in I_j$. We call the pair $(\vec{d}\:, e_2)$
\textbf{separated}, if exactly one of the two numbers 
$$\left\{ \frac{\Pi(\vec{d}\:)e_2}{k} \right\}\:,\ \ \Phi_j(\vec{d\:}, e_2)  $$
lies in a non-exceptional cell of depth $s_1$ or if they lie in two distinct non-exceptional cells of depth $s_1$. We call the pair $(\vec{d}\:, e_2)$ 
\textbf{exceptional}, if $\{\Pi(\vec{d}\:)e_2k^{-1}\}$ lies in an exceptional cell of depth $s_1$. Let $N_1(q)$ denote the number of separated pairs $(\vec{d}\:, e_2)$
with $\vec{d\:}\in \mathcal{A}(q)$ and
$N_2$ be the number of exceptional pairs $(\vec{d\:}, e_2)$.
\end{definition} 
\begin{lemma}\label{lem4444}
For all $\epsilon>0$ we have\\
(i) $N_1(q)\ll_\epsilon k^{D-1+\epsilon}q^{-1}\exp(2C_4s_1)$ \\
(ii) $N_2\ll_\epsilon k^{D+\epsilon}\exp(-(C_5-\epsilon)s_1)$.
\end{lemma}
\begin{proof}
 We first consider the number $N(q, j, \vec{d}\:)$ for a fixed $\vec{d}\:$ and 
 $e_2\in I_j$.\\
 By (2.1) the length of a non-exceptional cell of depth $s_1$ is $\leq \exp(-2C_4s_1)$. Thus the number of cells of depth $s_1$ is $\ll \exp(2C_4s_1)$.\\
 Let $(\vec{d}\:, e_2)$, $e_2\in I_j$ be separated. Then 
 $$\left\{ \frac{\Pi(\vec{d}\:)e_2}{q} \right\}\:\ \ \text{or}\ \ \left\{ \frac{\Pi(\vec{d}\:)(e_2-1)}{q} \right\}  $$
is the smallest or the largest number of the form $r/q$ of a cell of depth $s_1$.\\
Thus for fixed $\vec{d}$ and $j$  there are 
\[
\ll \exp(2C_4s_1)\tag{5.7}
\]
separated tuplets $(\vec{d}\:, e_2)$ with $e_2\in I_j$. The number of intervals 
$I_j\subset [k^{2v}, 2k^{2v})$ is 
\[
\ll k^{2v} q^{-1} \tag{5.8}
\]
The number of $\vec{d}$ with $\Pi(\vec{d}\:)\equiv l\bmod\: k$ is 
\[
\ll k^{D-1-2v+\epsilon}\:.\tag{5.9}
\]
The result (i) follows from (5.7), (5.8), (5.9), and the result (ii) from Lemma \ref{lem219} and the formula (5.9).
\end{proof}
\begin{definition}\label{def4545}
Let 
$$\mathcal{T}=\{ (\vec{d\:}, e_2)\::\: \vec{d}\in\mathcal{E}(v, \delta_1, \delta_2, k)\ \text{or}\ (\vec{d}\:, e_2)\ \text{separated or}\ (\vec{d}\:, e_2)\ \text{exceptional} \:\}$$
\end{definition}
\begin{lemma}\label{lem4545}
For all $\epsilon>0$ we have
$$\sum_{(\vec{d}\:, e_2)\in\mathcal{T}}\left| g\left(\frac{\Pi(\vec{d}\:)e_2}{k} \right) \right|\ll_\epsilon k^{D-1+\epsilon} \exp(2C_4s_1)+k^{D+\delta_1-\delta_2+\epsilon}+k^{D+\epsilon}\exp(-(C_5-\epsilon)s_1)\:.$$
\end{lemma}
\begin{proof}
This follows from Lemma \ref{lem4242}, the bound $q(l/k)=O(\log k)$, from Lemma \ref{lem4444} (i) by summing over $q$ and from Lemma
\ref{lem4444} (ii).
\end{proof}
For non-separated, non-exceptional pairs $(\vec{d}\:, e_2)$ we now replace the
terms $g(\Pi(\vec{d}\:)e_2k^{-1})$ by $g(\Phi_j(\vec{d\:}, e_2))$.
\begin{lemma}\label{lem4646}
Let $\vec{d}\in\mathcal{A}(q)$, $(\vec{d}\:, e_2)\not\in\mathcal{T}$. Then we have
$$\left| g\left( \frac{\Pi(\vec{d}\:)e_2}{k}  \right)-g(\Phi_j(\vec{d}\:, e_2)\right|\ll 
q^{-2}\exp(C_4s_1)+q2^{-s_1/2}\log k\:.$$
\end{lemma}
\begin{proof}
This follows from (5.8) and Lemma \ref{lem216b}.
\end{proof}
We now determine the total contribution of the tuplets $\vec{d}$ with appropriate
Diophantine approximation to the sum ${\textstyle\sum}_{1,1}^{(v)}$.
\begin{lemma}\label{lem4747}
For all $\epsilon>0$ we have:
\begin{align*}
\sum_{k^{\delta_1}\leq q< k^{\delta_2}}\sum_{\vec{d\:}\in \mathcal{A}(q)}\left|\sum_{\substack{e_2\::\: k^{2v}\leq e_2< 2k^{2v}\\ k^D\leq \Pi(\vec{d}\:)e_2<2k^D}}g\left( \frac{\Pi(\vec{d}\:)e_2}{k}\right)  \right| &\ll_\epsilon k^{D-2\delta_1+\epsilon}\exp(C_4s_1)+k^{D-2v+\delta_2+\epsilon }\\
&+k^{D+\epsilon}2^{-s_1/2}+k^{D-1+\epsilon}\exp(2C_4s_1)\\
&+k^{D+\delta_1-\delta_2+\epsilon}
+k^{D+\epsilon}\exp(-(C_5-\epsilon)s_1)\:.
\end{align*}
\end{lemma}
\begin{proof}
For $\vec{d\:}\in\mathcal{A}(q)$ we use the partition $\mathcal{J}=\bigcup_{j=0}^{j_0} I_j$ from Definition \ref{def4545} and the function $\Phi_j$. We have
\begin{align*}
 \tag{5.10}  \sum_{e_2\in I_j} g\left( \frac{\Pi(\vec{d}\:)e_2}{k}\right)&=\sum_{e_2\in I_j, e_2\not\in \mathcal{T}}\left( g\left( \Phi_j(\vec{d\:} ,e_2)+\left| g\left( \frac{\Pi(\vec{d}\:)e_2}{k}\right)-g(\Phi_j(\vec{d\:}, e_2))\right| \right) \right)\\
&+\sum_{e_2\in I_j, e_2\in \mathcal{T}}g(\Phi_j(\vec{d\:}, e_2))+O\left(\sum_{e_2\in I_j, e_2\in \mathcal{T}}|g(\Phi_j(\vec{d\:}, e_2))| \right)\\
&+O\left( \sum_{e_2\in I_j, e_2\in \mathcal{T}}  \left|g\left( \frac{\Pi(\vec{d}\:)e_2}{k}\right)\right|\right)\:.
\end{align*}
By Lemma \ref{lem4646} we have
\[
\sum_{e_2\in I_j, e_2\not\in \mathcal{T}}\left| g\left( \frac{\Pi(\vec{d}\:)e_2}{k}\right)-
 g( \Phi_j(\vec{d\:} ,e_2))\right| \ll q^{-1}\exp(C_4s_1)\log k\:.  \tag{5.11}
\]
Since the numbers $q\Phi_j(\vec{d\:}, e_2)$ form a complete residue system $\bmod\: q$ and because of the antisymmetry of $g$ we have
\[
\sum_{e_2\in I_j}g(\Phi_j(\vec{d\:}, e_2))=0\:.\tag{5.12}
\]
From (5.11) and (5.12) we obtain for $0\leq j\leq j_0-1$:
\[
\sum_{e_2\in I_j, e_2\not\in \mathcal{T}}  g\left( \frac{\Pi(\vec{d}\:)e_2}{k}\right) \ll q^{-1}\exp(C_4s_1)\log k+\sum_{e_2\in I_j, (\vec{d\:}, e_2)\in \mathcal{T}}\left| g\left( \frac{\Pi(\vec{d}\:)e_2}{k}\right) \right|\:.  \tag{5.13}
\]
For $j=j_0$ we obtain:
\[
\sum_{e_2\in I_j} g\left( \frac{\Pi(\vec{d}\:)e_2}{k}\right)\ll q\log k\:. \tag{5.14}
\]
Summing over all $O(k^{2v_0 }q^{-1})$  intervals $I_0, \ldots, I_{j_0-1}$ we
obtain from (5.13) and (5.14):
\begin{align*}
\sum_{e_2\in \mathcal{J}, e_2\not\in \mathcal{T}}  g\left( \frac{\Pi(\vec{d}\:)e_2}{k}\right)& \ll k^{2v+\epsilon }q^{-2}\exp(C_4s_1)+q\log k\\
&+\sum_{e_2\in \mathcal{J}\::\: (\vec{d\:}, e_2)\in \mathcal{T}} \left| g\left( \frac{\Pi(\vec{d}\:)e_2}{k}\right)\right|  
\end{align*}
Summing over all $O(k^{D-2v+\epsilon})$ tuplets $\vec{d\:}$ we obtain the result
by application of Lemma \ref{lem4545}.
\end{proof}
For a given $v$ we now choose the parameters $\delta_1, \delta_2, s_1$ and
the constants $C_4, C_5$ optimally. Writing $s_1=v_1\log k$ all the terms appearing in Lemmas \ref{lem4242} and \ref{lem4747} become powers of $k$, whose exponents are functions of these parameters, linear in each variable and thus monotonic.\\
By equating expressions with opposite monotonicity we obtain:
\begin{lemma}\label{lem4848}
Let $C_4$ be determined by
$$C_4\geq \frac{\sqrt{5}+1}{2}\ \text{and}\ \frac{1}{4}\:C_4-\log C_4-1+\log 2=\frac{1}{2}\:\log 2\:.$$
Let 
$$E(v)=v\left(1-\left(1+2\log 2\left(C_4+\frac{\log 2}{2}\right)^{-1}\right)^{-1}\right)\:.$$
Then for all $\epsilon>0$ we have:
$${\textstyle\sum}_{1,1}^{(v)}\ll_\epsilon k^{D-E(v)+\epsilon}\:.$$
\end{lemma} 
\vspace{7mm}
\section{The sum ${\textstyle\sum}_{1,2}$}
By (4.4) we have
\[
{\textstyle\sum}_{1,2}=\sum_{\substack{s\geq w, t\geq w\\ k^D\leq st\gamma< 2k^D \\ st<k^{4\delta_0+4v_0}}}\mu(\gamma)\sum(s, t)\:,   \tag{6.1}
\]
where
$$\sum(s, t)=\sum_{k^D\leq u\gamma<2k^D}\mu(\gamma)F(u)g\left(\frac{u\gamma}{k}\right)$$
with
$$F(u):=\sum_{\substack{(s,t)\in\mathbb{N}^2\\ st=u}}\sum_{\substack{d_1d_2=s\\ d_1\leq w}}\mu(d_1)\sum_{\substack{e_1e_2=t\\ e_1\leq w}}\mu(e_1)\:.$$
We observe that
$$|F(u)|\leq d_4(u)\ll_\epsilon u^\epsilon\:,\ \ \text{for all $\epsilon>0$}\:,$$
where $d_4$ denotes the divisor function of order 4.\\
Let $u^*>0$ be a constant to be determined later.\\
Let
\[
\mathcal{R}=\{ (U, V)\::\: U\geq 0,\ V\geq 0,\ D\log k\leq U+V< D\log k+\log 2, U\leq u^*\log k  \}\tag{6.2}
\]
We recursively define the sequence $(R_l)$ of regions that exhaust $\mathcal{R}$.\\
Let $\mathcal{R}_0$ be the union of squares with sides parallel to the coordinates axes, the coordinates of whose vertices are integer multiples of $2^{-v}=1$ and 
side lengths $2^{-v}=1$, that are completely contained in $\mathcal{R}$.\\
The recursion $l\rightarrow l+1$: Assume that $R_l$ has been defined. We let
$S_l$ be the union of squares of sides parallel to the coordinate axes of side
lengths $2^{-(l+1)}$, the coordinates of their vertices are integer multiples of $2^{-(l+1)}$, whose interiors are contained in the set $\mathcal{R}\setminus R_l$. We set
$$\mathcal{R}_{l+1}:=\mathcal{R}_l\cup S_l\:.$$
For a square $Q$ of $\mathcal{R}_l$ we define 
\[
\text{Exp}(Q):=\{ (u, \gamma)\::\: (\log u, \log\gamma)\in Q \}\:.\tag{6.3}
\]
We define $\Delta_u:=\Delta_u(Q)$, $\Delta_\gamma:=\Delta_\gamma(Q)$ by
\[
\text{Exp}(Q)=:[u_0, u_0(1+\Delta_u))\times[\gamma_0, \gamma_0(1+\Delta_\gamma))\:.\tag{6.4}
\]
We now define 
\[
l_0:=\max\{l\in\mathbb{N}\::\: \gamma_0\Delta_\gamma\geq k\ \text{for all squares $Q$ of $\mathcal{R}_l$}  \}\:.\tag{6.5}
\]
We consider the partial sum
\[
\sum(Q):=\sum_{\substack{u_0\leq u<u_0(1+\Delta_u)\\ \gamma_0\leq \gamma<\gamma_0(1+\Delta_\gamma)}}\mu(\gamma)F(u)g\left( \frac{u\gamma}{k} \right).\tag{6.6}
\]
We have
\begin{align*}
\left| {\textstyle\sum}_{1,2} (Q)\right|&\leq \left(\sum_{\gamma_0\leq \gamma<\gamma_0(1+\Delta_\gamma)}\mu(\gamma)^2 \right)^{1/2}\\
&\times\left( \sum_{\gamma_0\leq \gamma<\gamma_0(1+\Delta_\gamma)} \left(\sum_{u_0\leq u<u_0(1+\Delta_u)}F(u)g\left(\frac{u\gamma}{k}\right)\right)^2  \right)^\frac{1}{2}
\end{align*}
and
\begin{align*}
&  \sum_{\gamma_0\leq \gamma<\gamma_0(1+\Delta_\gamma)} \left(\sum_{u_0\leq u<u_0(1+\Delta_u)}F(u)g\left(\frac{u\gamma}{k}\right)\right)^2   \\
=&\sum_{u_0\leq u_1, u_2 \leq u_0(1+\Delta_u)}F(u_1)F(u_2) {\textstyle\sum}_{g}(Q, u_1, u_2)\:,
\end{align*}
where 
$$ {\textstyle\sum}_{g}(Q, u_1, u_2):= \sum_{\gamma_0\leq \gamma<\gamma_0(1+\Delta_\gamma)} g\left( \frac{u_1\gamma}{k} \right) g\left( \frac{u_2\gamma}{k} \right)\:.$$
We now fix $s_2\in\mathbb{N}$, $C_6\in\mathbb{R}$, to be determined later. We set
\[
 \mathscr{C}_{1, s_2}:=\{ x\in (0,1)\::\: q_{s_2}(x)\leq \exp(C_6s_2)  \} \tag{6.7}
\]
\[
 \mathscr{C}_{2, s_2}:=(0,1)\setminus  \mathscr{C}_{1, s_2}\tag{6.8}
\]
and define
\[
 {\textstyle\sum}_{g}^{(1)}(Q, u_1, u_2):=\sum_{\substack{\gamma_0\leq \gamma<\gamma_0(1+\Delta_\gamma)\\\left\{ \frac{u_i\gamma}{k} \right\}\in  \mathscr{C}_{1, s_2},\ i=1,2 }} g\left( \frac{u_1\gamma}{k} \right) g\left( \frac{u_2\gamma}{k} \right)   \tag{6.9}
\]
\[
 {\textstyle\sum}_{g}^{(2)}(Q, u_1, u_2):=\sum_{\substack{\gamma_0\leq \gamma<\gamma_0(1+\Delta_\gamma)\\\left\{ \frac{u_i\gamma}{k} \right\}\in  \mathscr{C}_{2, s_2},\ \text{$i=1$ or $2$} }} g\left( \frac{u_1\gamma}{k} \right) g\left( \frac{u_2\gamma}{k} \right)   \tag{6.10}
\]
For $\left\{ \frac{u_i\gamma}{k} \right\}\in  \mathscr{C}_{2, s_2}$ we use the trivial
estimate
$$g\left( \frac{u_i\gamma}{k} \right) <\log k\:.$$
The congruence
$$u_i\gamma \equiv l\bmod k$$
has $O(\gamma_0\Delta\gamma k^{-1})$ solutions, $\gamma\in [\gamma_0, \gamma_0(1+\Delta_\gamma))\:.$\\
Therefore we obtain:
\[
{\textstyle\sum}_{g_1}^{(2)}(Q) \ll\gamma_0\Delta_\gamma\exp(-(C_7-\epsilon)s_2)(\log k)^2\:, \tag{6.11}
\]
where $C_7=\frac{1}{2}\:C_6-  \log C_6-1+\log 2$.\\
We now apply Definition \ref{def28} to break up ${\textstyle\sum}_{g}^{(1)}(Q)$. We have
$${\textstyle\sum}_{g}^{(1)}(Q)={\textstyle\sum}_{g, 1}(Q)+{\textstyle\sum}_{g,2}(Q)\:,$$
where
$${\textstyle\sum}_{g, 1}(Q):= \sum_{\gamma_0\leq \gamma<\gamma_0(1+\Delta_\gamma)} g_{sm}\left( \frac{u_1\gamma}{k}, s_2 \right) g_{sm}\left( \frac{u_2\gamma}{k}, s_2 \right)$$
and
$${\textstyle\sum}_{g, 2}(Q):={\textstyle\sum}_{g, 2}^{(1)}(Q)+{\textstyle\sum}_{g, 2}^{(2)}(Q)+{\textstyle\sum}_{g, 2}^{(3)}(Q)$$
with
$${\textstyle\sum}_{g, 2}^{(1)}(Q)= \sum_{\gamma_0\leq \gamma<\gamma_0(1+\Delta_\gamma)} g_{sm}\left( \frac{u_1\gamma}{k}, s_2 \right) g_{sing}\left( \frac{u_2\gamma}{k}, s_2 \right)$$
$${\textstyle\sum}_{g, 2}^{(2)}(Q)= \sum_{\gamma_0\leq \gamma<\gamma_0(1+\Delta_\gamma)} g_{sing}\left( \frac{u_1\gamma}{k}, s_2 \right) g_{sm}\left( \frac{u_2\gamma}{k}, s_2 \right)$$
$${\textstyle\sum}_{g, 2}^{(3)}(Q)= \sum_{\gamma_0\leq \gamma<\gamma_0(1+\Delta_\gamma)} g_{sing}\left( \frac{u_1\gamma}{k}, s_2 \right) g_{sing}\left( \frac{u_2\gamma}{k}, s_2 \right)$$
By Lemma \ref{lem211} we have:
$$g_{sing}\left( \frac{u_i\gamma}{k}, s_2 \right)=O(2^{-s_2/2}\log k)\:.$$
We obtain
\[
{\textstyle\sum}_{g, 2}^{(i)}(Q)=O(\gamma_0\Delta\gamma\:2^{-s_2/2}\log k)\ \ (i=1, 2, 3) \tag{6.12}
\]
For the estimate of ${\textstyle\sum}_{g, 2}^{(1)}(Q)$ we replace the terms 
$$g_{sm}\left( \frac{u_1\gamma}{k} \right) g_{sm}\left( \frac{u_2\gamma}{k} \right)$$
by their mean-values
\[
M^{(1)}(u_1, u_2, \gamma):=k\int_{\frac{\gamma}{k}-\frac{1}{2k}}^{\frac{\gamma}{k}+\frac{1}{2k}} g_{sm}(u_1x)g_{sm}(u_2x)dx\:. \tag{6.13}
\]
From the definition of $ \mathscr{C}_{1, s_2}$ we see that $\left\{\frac{u_1\gamma}{k}\right\}$ resp. $\left\{\frac{u_2\gamma}{k}\right\}$ are placed in cells $\mathcal{C}_1$ resp. $\mathcal{C}_2$ of depth $s_2$ with 
$$q_2(\mathcal{C}_i)\leq \exp(C_6s_2)\:.$$
By Lemma \ref{lem216b} we have
\begin{align*}
  \tag{6.14} &k \int_{\frac{\gamma}{k}-\frac{1}{2k}}^{\frac{\gamma}{k}+\frac{1}{2k}} \left|g_{sm}\left( \frac{u_1\gamma}{k}, s_2 \right) g_{sm}\left( \frac{u_2\gamma}{k}, s_2 \right)  
-g_{sm}(u_1x, s_2) g_{sm}(u_2x, s_2)\right|dx \\
&\ \ \ll_\epsilon u_0\exp(C_6s_1) k^{-1}\:.
\end{align*}
We now compare the mean-values $M^{(1)}(u_1, u_2, \gamma)$ and 
$$M^{(2)}(u_1, u_2, \gamma):=k \int_{\frac{\gamma}{k}-\frac{1}{2k}}^{\frac{\gamma}{k}+\frac{1}{2k}}  g(u_1x)g(u_2x)dx\:.$$
We have
\begin{align*}
\tag{6.15} &\sum_{\gamma\in[\gamma_0, \gamma_0(1+\Delta_\gamma))}\int_{\frac{\gamma}{k}-\frac{1}{2k}}^{\frac{\gamma}{k}+\frac{1}{2k}}(  g_{sm}(u_1x, s_2)g_{sm}(u_2x, s_2)-g(u_1x)g(u_2x)dx\\
&\ll \sum_{\gamma\in[\gamma_0, \gamma_0(1+\Delta_\gamma))}\int_{\frac{\gamma}{k}-\frac{1}{2k}}^{\frac{\gamma}{k}+\frac{1}{2k}}| g_{sing}(u_1x)|\:|g_{sm}(u_2x)|\\
&\ \ \ \ \ \ \ \ \ \ \ \ \  +|g_{sm}(u_1x)|\:|g_{sing}(u_2x)|+|g_{sing}(u_1x)|\:|g_{sing}(u_2x)|\:dx\\
&\leq \left( \int_{\frac{\gamma_0}{k}}^{\frac{\gamma_0}{k}(1+\Delta_\gamma)}|g_{sing}(u_1x, s_2)|^2dx\right)^{1/2}\left( \int_{\frac{\gamma_0}{k}}^{\frac{\gamma_0}{k}(1+\Delta_\gamma)}|g(u_2x)|^2dx\right)^{1/2}\\
&+\left( \int_{\frac{\gamma_0}{k}}^{\frac{\gamma_0}{k}(1+\Delta_\gamma)}|g(u_1x)|^2dx\right)^{1/2}\left( \int_{\frac{\gamma_0}{k}}^{\frac{\gamma_0}{k}(1+\Delta_\gamma)}|g_{sing}(u_2x, s_2)|^2dx\right)^{1/2}\\
&+\left( \int_{\frac{\gamma_0}{k}}^{\frac{\gamma_0}{k}(1+\Delta_\gamma)}|g_{sing}(u_1x)|^2dx\right)^{1/2}\left( \int_{\frac{\gamma_0}{k}}^{\frac{\gamma_0}{k}(1+\Delta_\gamma)}|g_{sing}(u_2x)|^2dx\right)^{1/2}\\
&\ll_{\epsilon} \gamma_0\Delta_\gamma k^{-1+\epsilon} g^{s_2}
\end{align*}
by Lemma \ref{lem213}.\\
We finally obtain from (6.11), (6.12), (6.14) and (6.15).
\[
{\textstyle\sum}_{g}^{(1)}(Q) \ll_{\epsilon}  k^{\epsilon} u_0\gamma_0\Delta_u\Delta_\gamma\left(\exp\left(-\frac{C_7}{2}s_2\right)+k^{-1/2}\exp\left(\frac{C_6s_2}{2} \right)u_0^{1/2} + g^{s_2/2}\right)\:.  \tag{6.16}
\]
\begin{lemma}\label{lem5151}
Let $z=(u_1, u_2)$. Then
$$\int_0^1 g(u_1x)g(u_2x)dx\ll \frac{z^2}{u_1^2u_2^2}\:.$$
\end{lemma}
\begin{proof}
By Parseval's equation we have:
$$\int_0^1g(u_1x)g(u_2x)dx=\sum_{(n_1, n_2)\::\: n_1u_1=n_2u_2}\frac{d(n_1n_2)^2}{n_1^3 n_2^3}\:.$$
\end{proof}
We conclude the estimate of ${\textstyle\sum}_{1,2}$ by:
\begin{lemma}\label{lem5252}
\begin{align*}
&{\textstyle\sum}_{1,2}\ll_\epsilon k^{D+\epsilon} \Bigg(k^{-1/2} u_0^{1/2}\exp\left(\frac{C_6s_2}{2}\right)+\exp\left(-\frac{C_7}{2}\: s_2\right)\\
&\ \ +2^{-s_2/4}+g^{s_2/2}+u_0^{-1/2} \Bigg) \:.
\end{align*}
\end{lemma}
\begin{proof}
This is obtained by summing the estimates (6.16) and from Lemma \ref{lem5151}
over all squares $Q$ of $\mathcal{R}_{l_0}$ and by estimating the contribution from the pairs $(u, \gamma)\not\in \mathcal{R}_{l_0}$ trivially.\\
\end{proof}
We now again choose the variables $C_6$ and $s_2$ optimally which leads to 
the following lemma (with $C_6=C_4$, $C_7=C_5=\frac{\log 2}{2}$):
\begin{lemma}\label{lem5353}
$${\textstyle\sum}_{1,2}\ll_\epsilon k^{D-H(u^*)+\epsilon}\:, $$
where
$$H(u^*):=\frac{1}{2}-\left(\frac{1}{2}+\frac{C_4}{\log 2}\right) u^*\:.$$
\end{lemma}
\vspace{7mm}
\section{Conclusion of the proof}
\begin{lemma}\label{lem6161}
$${\textstyle\sum}_{2}\ll_\epsilon k^{D-1+4\delta_0+\epsilon}  $$
and
$${\textstyle\sum}_{3}\ll_\epsilon k^{D-1+4\delta_0+\epsilon}  $$
\end{lemma}
\begin{proof}
This follows from the estimate
$$g\left(\frac{l}{k}\right)\ll \log k\:.$$
\end{proof}
We now set $u^*:=4v_0+4\delta_0$ and choose $\delta_0$ arbitrarily small.\\
We then determine $v_0$ by the equation
$$E(v_0)=H(u^*)\:.$$
Our Theorem then follows from Lemma \ref{lem4848}, Lemma \ref{lem5353},
and Lemma \ref{lem6161}.

\end{document}